\documentclass[11pt]{article}
\usepackage{amsmath,amssymb,amsthm}
\usepackage{graphicx}
\newtheorem{Thm}{Theorem}
\newtheorem{Lem}[Thm]{Lemma}
\newtheorem{Prop}[Thm]{Proposition}
\newtheorem{Cor}[Thm]{Corollary}
\newtheorem{Conj}{Conjecture}
\newtheorem{Rque}{Remark}
\newtheorem{Def}{Definition}
\newtheorem{Exam}{Example}

\def\rit{{\mathbb R}}
\def\cit{{\mathbb C}}
\def\nit{{\mathbb N}}
\def\qit{{\mathbb Q}}
\def\zit{{\mathbb Z}}
\def\pit{{\mathbb P}}

\def\0{{\mathcal O}}
\def\W{{\mathcal W}}
\def\U{{\mathcal U}}
\def\K{{\mathcal K}}
\def\C{{\mathcal C}}
\def\S{{\mathcal S}}
\def\T{{\bf T}}
\def\P{{\mathfrak P}}

\def\h{{\mathfrak h}}

\def\locus{\mathop{\rm locus}\nolimits}

\def\RatCurves{\mathop{\rm RatCurves}\nolimits}

\title{Minimal rational curves on complete toric manifolds}
\author{Baohua Fu and Jun-Muk Hwang}

\begin{document}
\maketitle
\section{Introduction}

For a complete uniruled smooth variety $X$,  let $\RatCurves^n(X)$ be
the normalized space of rational curves on $X$. For an irreducible
component $\K$ of $\RatCurves^n(X)$, let $\rho: \U\to \K$ and $\mu: \U\to X$ be the associated
universal family morphisms. An irreducible component $\K$
of $\RatCurves^n(X)$ is called {\em a dominating component} if $\mu$ is dominant and {\em a minimal  component} if furthermore for a general point
$x \in X$, the variety $\mu^{-1}(x)$ is complete. Members of a minimal component are called {\em minimal
rational curves}. The {\em degree} of $\K$ is the degree of the intersection of  $-K_X$ with any member in $\K$.
 For a fixed minimal component $\K$ and a general point $x \in X$,
we  define the tangent map: $\tau_x: \K_x := \rho(\mu^{-1}(x)) \dasharrow \pit(T_x X)$  by
$\tau_x(\alpha) = \pit(T_xC)$, where $C= \mu(\rho^{-1} (\alpha))$ is smooth at $x$.  We denote by $\C_x$ the closure of
the image of $\tau_x$ in $\pit(T_xX)$, which is called the {\em variety of minimal rational tangents} (VMRT for short) of
$\K$ at the point $x \in X$.

It turns out that the projective geometry of $\C_x \subset \pit(T_xX)$ encodes a lot of the geometrical properties on $X$,
which can be a useful tool in solving a number of problems on uniruled varieties (see the surveys \cite{Hw1}, \cite{Mok}).  Thus for a given $X$, it is worthwhile to determine $\C_x$.
This has been worked out for many examples when $X$ has Picard number 1 (loc. cit.). However, not many cases
with large Picard number have been investigated. The main result of this paper, Corollary \ref{VMRT},
gives a description of $\C_x$ for a complete toric manifold.
 This implies that the minimal components of $\RatCurves^n(X)$ correspond bijectively to some special primitive collections
(Theorem \ref{primitive}), which can be easily read off from the fan  defining $X$.
As an application, we get examples of complete toric manifolds which do not have any minimal component
 (Example \ref{e.1}). Also, we will be able to classify toric Fano manifolds admitting a minimal component  of degree $n = \dim X$ (Proposition \ref{degn}). Finally, in Section 5,  we propose a conjectural upper bound for $\rho_X(\dim \C_x +1)$ when $X$ is a
toric Fano manifold, which is motivated by a conjecture of Mukai (see \cite{Ca2}).
\vspace{0.3cm}

{\em Acknowledgements:} The first named author would like to thank
KIAS (Seoul) and the University of Hong Kong for their
hospitality. The second author was supported by the Korea Research
Foundation Grant
 (KRF-2006-341-C00004).

\section{Some general results}

Here we collect a couple of general facts for which we do not have good references.
The first one is more or less obvious.

\begin{Lem}\label{lem1} Let $X$ be a complete variety on which a connected algebraic group $G$  acts with
 an open orbit $X_0 \subset X$. Suppose the stabilizer ${\rm Stab}_G(x)\subset G$ of a point $x \in X_0$ is
connected. Then for any dominating component $\K$ of ${\rm RatCurves}^n(X)$, the subvariety $\K_x \subset \K$
is irreducible. \end{Lem}

\begin{proof}
$G$ acts on the universal family $\rho: \U \to X$ of $\K$. This action descends to the  finite
morphism $\rho': \U' \to X$ obtained by the Stein factorization of $\rho$. Since the stabilizer ${\rm Stab}_{G}(x)$ is connected, it fixes a point $y \in \rho'^{-1}(x)$, i.e., ${\rm Stab}_G(y) = {\rm Stab}_G(x)$.
It follows that $\U'$ contains an open subset isomorphic to $X_0$, i.e., $\rho'$ is birational. This means that $\K_x$ is irreducible. \end{proof}

For the next one, we need some notation. Let $G$ be a connected algebraic group.
For an irreducible closed algebraic subvariety $S \subset G$   which contains the identity $e \in G$, let $[S]$ be the subgroup of $G$ generated by elements in $S$ and
let $\langle S \rangle$ be the smallest closed algebraic subgroup of $G$ containing $S$. Clearly, $[S] \subset \langle S \rangle.$
\begin{Prop}\label{general}
Let $G$ be a connected commutative algebraic group over the complex numbers.
If $S$ is a closed irreducible algebraic subvariety of  $G$  contained in an (not necessarily closed) analytic subgroup $H$ of $G$, then
$\langle S \rangle = [S] \subset H$.
\end{Prop}
\begin{proof}
We define $\W$ to be the holomorphic foliation on $G$ given by the cosets of $H$.
Let $\S$ be the family of subvarieties of $G$ obtained by $G$-translates of $S$.
Note that our assumption implies that every member of $\S$ is contained in a leaf of $\W$.

Let $Z \subset G$ be a closed irreducible subvariety lying in a single leaf $W$ of $\W$.
Let $Z-S:=\{z-s | z\in Z, s\in S\}$, which is an irreducible closed subvariety of $G$.
We define $I(Z)$ to be the image of the group action: $(Z-S) \times S \to G$, which is again an irreducible closed subvariety of $G$.
Note that $I(Z)$ is the union of members of $\S$ intersecting $Z$, which implies that $I(Z) \subset W$.
Define inductively $I^{i+1}(Z)= I(I^i(Z))$.  We have $I^{i}(Z) \subset I^{i+1}(Z) \subset W$.
This implies that  $I^n(Z)=I^{m}(Z)$ for $m \geq n = \dim G$.

 Note that by the definition, if $Z$ is contained in $[S]$, then
$I(Z)$ is contained in $[S]$. Thus we have $I^m(S) \subset [S]$ for any $m \in \nit$. On the other hand,  each element of $[S]$ is contained in
some $I^{m}(S),$ hence $[S] \subset I^n(S).$  Consequently, $[S] = I^n(S)$ is an algebraic variety, i.e.,
$[S]$ is an algebraic subgroup. It follows that
 $\langle S \rangle = [S]=I^n(S) \subset W$.
\end{proof}

\begin{Cor}\label{dim}
In Proposition \ref{general}, suppose there exists a Lie subalgebra $\h$ of the Lie algebra of $G$
 such that $S \subset \exp(\h)$, then $\dim \langle S \rangle \leq \dim \h$.
\end{Cor}

 Note that Corollary \ref{dim} is a generalization of Lemma I.3.2 in \cite{Zak} for abelian varieties. Our proof is completely different from Zak's.

\section{Varieties of minimal rational tangents on a complete toric manifold}

From now on, let $X$ be a complete toric manifold of dimension $n$ and $\T \subset X$ be the open orbit of
the torus $(\cit^*)^n$.
A rational curve $C \subset X$ is called {\em a standard curve} if under the normalization $\nu: \pit^1 \to C$, we have
$\nu^*(TX) \simeq \0(2) \oplus \0(1)^p \oplus \0^{n-p-1}$. It is easy to see that the deformations of $C$
 correspond to a dominating component $\K^C$ of ${\rm RatCurves}^n(X).$  As in Introduction, for $x \in \T$, we denote by $\K^{C}_x$ the family
 of deformations of $C$ passing through
$x$. Then $\dim \K^C_x = p$ and $\K^C_x$ is irreducible by Lemma \ref{lem1}.
Just as before, we can define a variety of tangents $VT_x^{C} \subset \pit(T_xX)$ by taking the closure of tangents at $x$
of curves in $\K_x^C$ which are smooth at $x$.
\begin{Thm}\label{standard}
Let $X$ be a complete toric manifold and $C \subset X$ a standard curve. Then the variety of tangents $VT_x^{C}$ of
$\K^{C}$ at a  $x \in \T$ is a  linear subspace in $\pit(T_xX)$.
\end{Thm}
\begin{proof}
Let $D = X \setminus \T$ be the boundary divisor.
It is well-known that $\Omega_X(\log D) \simeq \0_X^{\oplus n}$ (Prop. 3.1 \cite{Oda}).
Given $x \in \T$, we may identify $\T$ with the
 torus such that $x$ corresponds to the identity $e \in (\cit^*)
 ^n$.
For an irreducible and reduced curve $C$ passing through $e$, let $\langle C \rangle$ be the smallest toric subvariety of $X$ containing $C$. Then $$\langle C \rangle \cap \T = \langle C \cap \T \rangle$$ where
the right hand side is in the sense of Proposition \ref{general}. Let $V$ be the subspace of $H^0(X, \Omega_X(\log D)) \simeq \cit^n$ consisting of vectors which annihilate
the tangent vectors $TC$ along $C$.  The vector
space $V$ contains the space $H^0(C, T_X^*|_C) \subset H^0(C, \Omega_X(\log D)|_C)=H^0(X,  \Omega_X(\log D))$.
As $C$ is a standard curve, the space  $H^0(C, T_X^*|_C)$ has dimension $n-p-1$, thus
we get $\dim V \geq n-p-1$.
By Corollary \ref{dim},  this implies  $\dim \langle C \rangle \leq n- \dim V  \leq p+1$.

Let $\K^o \subset \K_e^C$ be the open subset consisting of standard curves. Denote by $\locus(\K^o)$ the closure of the union
of members of $\K^o$. Then $\locus(\K^o)$ is a $(p+1)$-dimensional constructible
subset   of $X$.  Consider the $p$-dimensional family of toric subvarieties $\{\langle C_t \rangle\}_{t\in \K^o}$ which pass
through the fixed point $e$. Since there is no positive-dimensional family of  toric subvarieties fixing a point,
we have $\langle C_t \rangle = \langle C_{t'} \rangle$ for two general points $t, t' \in \K^o$.
Thus $\locus(\K^o) \subset \langle C_t \rangle$ for some general $t \in \K^o$. Since
$\dim(\locus(\K^o)) = p+1$, while $\dim (\langle C_t \rangle) \leq p+1$,   we have   $\langle C_t \rangle = {\locus}(\K^o)$ in a neighborhood of $e$. In particular, we have $VT^{C}_e = \pit(T_e \langle C_t \rangle)$ which is linear since the toric subvariety  $\langle C_t \rangle$ is smooth at $e$.
\end{proof}
\begin{Rque}
The same argument works for singular complete toric varieties if one assumes that the standard curve $C$ is contained in the smooth locus of $X$.
\end{Rque}
\begin{Rque}
Let  $Y^n$ be a smooth projective variety dominated by a  toric
variety $X$, i.e. there exists a surjective morphism $\pi: X \to
Y$. If $\rho_Y=1$, then by Theorem 1 \cite{OW}, $Y$ is isomorphic
to $\pit^{n}$. In Section 1 \cite{OW}, they raised  the question
whether $Y$ is always toric. The answer is not always affirmative.
In fact, let $E$ be a vector bundle over a projective toric
manifold $X$, take an ample line bundle $L^{-1}$  such that
$E\otimes L^{-1}$ is globally generated. This gives a surjective
map $L^{\oplus d} \to E$ for some integer $d$. Then we obtain a
surjective map $\pit^{d-1} \times X \to \pit(E)$, while $\pit(E)$
is not necessarily a toric variety. We do not know any other
example of projective manifolds dominated by a toric variety. Note
that we may assume that $\pi$ is generically finite by taking
Stein factorization and resolution. Then, using a similar argument
as in the proof of the precedent theorem and that of Prop. 2
\cite{HM}, we can show that for any standard curve $C$ in $Y$, the
variety of tangents $VT_y^C$ at a general point $y \in Y$ is a
disjoint union of linear subspaces. Thus $Y$ contains projective
spaces with trivial normal bundles. But these projective spaces do
not always give rise  to a projective bundle structure on $Y$.
\end{Rque}

\begin{Cor}\label{VMRT}
Let $X$ be a complete toric manifold and let $\K$ be a minimal component of degree $p+2$. The variety of minimal rational tangents $\C_x$ at
a general point $x\in X$ is an irreducible linear subspace. The locus $\locus(\K_e)$ of curves in $\K$ passing through the identity $e \in \T$ is a toric
subvariety in $X$, which is isomorphic to $\pit^{p+1}$ with trivial normal bundle.
\end{Cor}
\begin{proof}
As is well-known, a general member of $\K$ is a standard curve. Thus Theorem \ref{standard} implies that $\C_x$ is an irreducible linear subspace for $x \in X$ general.
As shown in \cite{Ara} (Lemma 3.3), $\locus(\K_e)$ is an immersed $\pit^{p+1}$ with trivial normal bundle.
By the proof of Theorem \ref{standard}, the subvariety $\locus(\K_e)$ is equal to $\langle C \rangle$ for a general curve $C$ in $\K_e$, thus it is
a toric subvariety in $X$. Being a toric subvariety, the variety
$\locus(\K_x)$ is itself normal, thus  it is smooth and isomorphic to $\pit^{p+1}$.
\end{proof}
\begin{Cor}\label{projbundle}
Let $X$ be a complete toric manifold and let $\K$ be a minimal
component of $\RatCurves^n(X)$ of degree $p+2$. Then there exist
an open subset $X_0$ containing the torus $\T$, a smooth variety
$U$ and a $\pit^{p+1}$-bundle structure $\phi_0: X_0 \to U$ such
that any curve in $\K$ meeting $X_0$ is a line on a fiber of
$\phi_0$.
\end{Cor}
\begin{proof}
By Theorem 1.1 \cite{Ara}, there exists such an open subset $X_0$, which contains the image of the torus action $(\cit^*)^n \cdot \pit^{p+1}$.
In particular, $X_0$ contains $\T$.
\end{proof}

This corollary implies that every projective toric manifold contains an open dense subset which is a projective bundle. This seems
to be a new result. In general, one cannot hope that $\phi_0$ can be  extended to a morphism of $X$ in codimension 1,
as shown by the blow-up of $\pit^2$ at two points.

\section{Combinatorial description of minimal rational curves}
We are now going to relate minimal components on $X$ to combinatorial data of the fan corresponding to $X$. The basic results on
toric varieties can be found in \cite{Oda}. Recall that $X$ is described by a finite fan $\Sigma$ in the vector space $N_\qit = N \otimes_\zit \qit$, where
$N$ is a free abelian group of rank $n=\dim X$. As $X$ is smooth and complete,  the support of $\Sigma$ is the whole space $N_\qit$ and
every cone in $\Sigma$ is generated by a part of a basis of $N$.  For any $i$, we denote by $\Sigma(i)$ the
set of all $i$-dimensional cones in $\Sigma$. For each  $\sigma \in \Sigma(1)$, we take a primitive generator of $\sigma \cap N$.
We denote by $G(\Sigma)$ the set of all such generators, which is in bijection with all 1-dimensional cones in $\Sigma$.
The Picard number $\rho_X$ of $X$ is given by $\sharp G(\Sigma) - n$ (Cor. 2.5 \cite{Oda}).

\begin{Def}[\cite{Bat}]
i) A non-empty subset $\P=\{x_1, \cdots, x_k\}$ of $G(\Sigma)$ is called a {\em primitive collection} if for any $i$, the elements of
$\P \setminus \{x_i\}$ generate a $(k-1)$-dimensional cone in $\Sigma$, while $\P$ does not generate a $k$-dimensional cone in $\Sigma$.

ii) For a primitive collection $\P=\{x_1, \cdots, x_k\}$ of $G(\Sigma)$, let $\sigma(\P)$ be the unique cone in $\Sigma$ which contains $x_1+\cdots+x_k$ in its interior. Let $y_1, \cdots, y_m$ be generators of $\sigma(\P)$, then there exists a unique equation with $a_i \in \zit_{>0}$:
$$
x_1+ \cdots+ x_k = a_1 y_1 +\cdots a_m y_m.
$$
This is the {\em primitive relation} associated to $\P$. The {\em degree} of $\P$ is $\deg(\P) = k-\sum_i a_i$.
The {\em order} of $\P$ is $k$.
\end{Def}

\begin{Thm}\label{primitive}
Let $X$ be a complete toric manifold of dimension $n$. Then there is a bijection between minimal components of degree $k$ on $X$ and
primitive collections $\P=\{x_1, \cdots, x_k\}$ of $G(\Sigma)$ such that $x_1+\cdots+x_k=0$.
\end{Thm}
\begin{proof}
If $\K$ is a minimal component in $\RatCurves^n(X)$ of degree $k$, by Corollary \ref{VMRT}, there exists a smooth toric subvariety $P$ in $X$
which is isomorphic to $\pit^{k-1}$. Let $\Sigma' \subset \Sigma$ be the sub-fan which gives $P$. Let $\P:=G(\Sigma')=\{x_1, \cdots, x_k\}$. Then for any  $x_i \in \P$, the elements in $\P \setminus \{x_i\}$ generate a $(k-1)$-dimensional cone and we have $x_1+\cdots+x_k=0$,
which implies in particular that
$\P$ does not generate a $k$-dimensional cone in $\Sigma$ since every cone of $\Sigma$ is generated by a part of a basis of $N$. This implies that $\P$ is a primitive collection of $G(\Sigma)$.

Now assume $\P=\{x_1, \cdots, x_k\}$ is a primitive collection such that $x_1+\cdots+x_k=0$.
Let $\Sigma'$ be the sub-fan of $\Sigma$ determined by $\P$. Then $\Sigma'$ defines a toric subvariety $P$ of $X$ which is isomorphic to $\pit^{k-1}$.
Let $D = \sum_{i=1}^d D_i$ be the decomposition of the boundary divisor $D=X \setminus \T$ into irreducible components,
where $d=\rho_X+n$ is the number of elements in $G(\Sigma)$.
 We have
a  Euler-Jaczewski sequence (see for example Section 2.1 \cite{OW}) as follows:
$$
0 \to \Omega_X^1 \to \oplus_{i=1}^d \0_X(-D_i) \to \0_X^{\oplus \rho_X} \to 0.
$$
As $P$ is primitive, there is no other elements of $G(\Sigma)$ in the vector space $\qit x_1 +\cdots +\qit x_k$.
This shows that $\0_X(-D_i)|_P \simeq \0_P$ if $D_i$ is not a divisor corresponding to one of the one dimensional cones
in $\Sigma$ generated by $x_1, \cdots, x_k$.
The precedent exact sequence gives
$$
0 \to \Omega_X^1|_P \to \0_P(-1)^{k} \oplus \0_P^{\oplus \rho_X+n-k} \to \0_P^{\oplus \rho_X} \to 0.
$$

Combining with the usual Euler sequence for $P \simeq \pit^{k-1}$, this implies
that the normal bundle of $P$ in $X$ is trivial.  Take a line $C$ in $P$, then its normal bundle in $X$ is isomorphic to $\0(1)^{k-2} \oplus \0^{n-k+1}$,
i.e. $C$ is a standard curve.
Using the action of $(\cit^*)^n$, we get an irreducible
family of rational curves $\K$, which is a minimal component in $\RatCurves^n(X)$.
\end{proof}

When $X$ is a projective toric manifold, there always exists  a minimal component in $\RatCurves^n(X)$. Theorem \ref{primitive} has the following
corollary, which has been firstly proved by Batyrev (Prop. 3.2 \cite{Bat}) in a completely different way.
\begin{Cor}
Let $\Sigma$ be a fan which defines  a projective toric manifold $X$. Then there exists a primitive collection $\P=\{x_1, \cdots, x_k\}$
such that $x_1+\cdots+x_k=0$.
\end{Cor}
\begin{Exam}\label{e.1}\upshape
The assumption of projectivity of $X$ in the precedent corollary is important as shown by the following example in Section 2.3 \cite{Oda}.
Let $e_1, e_2, e_3$ be a basis of $\zit^3$. Let $$e_4=-e_1-e_2-e_3, \;  e_5=-e_1-e_2, \;  e_6=-e_2-e_3, \; e_7=-e_1-e_3.$$ Let $\Sigma$ be the complete regular fan in $\rit^3$ obtained by
joining $0$ with the simplices of the   triangulated tetrahedron in the displayed figure.
We have $e_1+e_2+e_5=0$, however the set $\{e_1, e_2, e_5\}$ is not a primitive collection,
since the cone generated by $e_2, e_5$ is not in $\Sigma$. Similarly we see that
$\{e_2, e_3, e_6\}, \{e_1, e_3, e_7\}$ are not primitive collections. This implies that there is no minimal component in $\RatCurves^n(X)$. This gives another way to see that the toric variety $X$ defined by $\Sigma$ is smooth complete but non-projective. The sub-fan in $\Sigma$ generated by $e_1, e_2, e_5$ gives a toric subvariety which is
isomorphic to $\pit^2\setminus \{pt\}$.
If we denote by $C$ the invariant curve corresponding to the cone generated by $e_1, e_3$, then its
cohomology class is given by $e_1+e_2+e_5=0$, which implies that its normal bundle in $X$ is given by
$\0(1) \oplus \0$, i.e. $C$ is a standard curve. By Theorem \ref{standard}, the variety of tangents $VT_x^C$
determined by $C$ is isomorphic to $\pit^1$, while all members of $\K_x^C$ lie in an open
set in $\pit^2 \setminus \{pt\}$. If we denote by $\pi: Y \to X$ the blow-up of $X$ along the invariant curve $C$,
then $Y$ is still non-projective since $C$ deforms in $X$ (Prop. 2 \cite{Bon}). The fan $\Sigma(Y)$ of $Y$
has a new element $e_0=e_1+e_3 = -e_7$. Thus the non-projective variety $Y$ has
a unique minimal component and its VMRT is just one point.

If one blows up $X$ along the invariant curve corresponding to the cone generated
by $e_3$ and $e_7$, we obtain a projective variety $X'$ (cf. Section 2.3 \cite{Oda}).
The fan $\Sigma'$ of $X'$ has a new element $e_8=-e_1$ in $G(\Sigma')$. This implies that
there exists a unique minimal component in $\RatCurves^n(X')$ and its VMRT is just a point.

\includegraphics[scale=0.55]{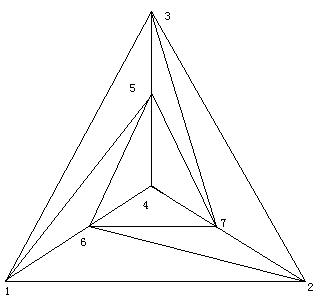}

\end{Exam}

As an application of Theorem \ref{primitive}, we give an upper bound for the number of minimal components in $\RatCurves^n(X)$.
\begin{Prop}
Let $X$ be a complete toric manifold of dimension $n$ and Picard number $\rho_X$. For an integer $p$, we denote by $n_p$ the number of
minimal components in $\RatCurves^n(X)$ of degree $p+2$. Then we have

i)   $$\sum_{p=0}^{n-1} n_p (p+2) \leq n+\rho_X.$$

ii)  If $n_p$ and $n_q$ are non-zero for some integers $p$ and $q$, then $p+q \leq n-2$.

iii) If $p \geq (n-1)/2$, then $n_p \leq 1$.
\end{Prop}
\begin{proof}
All three statements follow from the proof of Prop. 2.2
\cite{Hw2}, where it is shown that two linear subspaces in
$\pit(T_xX)$, which are  components of  VMRTs, have empty
intersection in $\pit(T_xX)$ for $x\in X$ general. For the claim
i), a more combinatorial proof can be given as follows. Suppose we
have two primitive collections $\P_1=\{x_1, \cdots, x_{k+1}\}$ and
$\P_2=\{y_1, \cdots, y_{h+1} \}$ such that $x_1+\cdots+x_{k+1} =
y_1+\cdots+y_{h+1}=0$. If $\P_1 \cap \P_2$ is non-empty, we may
assume $x_{k+1}=y_{h+1}$, then we get $x_1+\cdots+x_k =
y_1+\cdots+y_h$. As $x_1, \cdots, x_k$ and $y_1, \cdots, y_h$
generate  cones in $\Sigma$, this implies that the two cones are
the same, thus $\P_1=\P_2$. Now the claim i) follows from Theorem
\ref{primitive} and the fact that the number of elements in
$G(\Sigma)$ is equal to $n+ \rho_X$.
\end{proof}
\begin{Rque}\upshape
Even in dimension 2, one can construct many examples where the inequality in claim i) is an equality. If one restricts
to toric Fano manifolds, the inequality in i) becomes an equality for products of
copies of $S_3$ with projective spaces, where $S_3$ is the
the blow-up of $\pit^2$ at three general points. In Example 4.7 \cite{Sat}, Sato constructed a toric Fano 4-folds with $\rho = 5$ by blowing up $\pit^2 \times \pit^2$, for which the inequality i) becomes an equality. It seems to be subtle to classify those such that i) becomes an equality.
\end{Rque}

Another application of Theorem \ref{primitive} is the following.
Recall that if $X$ has  a minimal component of degree $n+1$, then $X \simeq \pit^n$ (\cite{CMS}).
The following proposition settles the next case when $X$ is a toric Fano manifold. Recall that for a toric Fano manifold, every element in $G(\Sigma)$ is a primitive vector in $N$, and $G(\Sigma)$ is the set of vertices of a polytope  $Q$ and each facet of $Q$ is the convex hull of a basis of $N$.
\begin{Lem}[Lem. 3.3 \cite{Cas}]\label{2prim}
Assume $X$ is a toric Fano manifold. If $\Sigma$ has two different primitive relations $x+y=z$ and
$x+w=v$, then $w=-x-y$ and $v=-y$. Therefore there are at most two primitive collections
of order 2 and degree 1 containing $x$, and the associated primitive relations are $x+y=(-w)$ and
$x+w = (-y)$.
\end{Lem}

\begin{Prop}\label{degn}
Let $X$ be a toric Fano manifold of dimension $n \geq 3$ which admits a minimal component  of degree $n$.
Then $X$ is isomorphic to  $\pit^{n-1} \times \pit^{1}$, $\pit(\0_{\pit^1}^{\oplus{n-1}} \oplus \0_{\pit^1}(1))$ or a blow-up
of a $\pit^{n-2}$ on $\pit^{n-1} \times \pit^{1}$. In particular, we have $\rho_X \leq 3$.
\end{Prop}
\begin{proof}
Note  that if $z_1, z_2 \in G(\Sigma)$ are two elements such that the two dimensional cone generated by them is not in $\Sigma$, then
 $\{z_1, z_2\}$ is a primitive collection.  The Fano condition on $X$ implies that either $z_1+z_2=0$ or
 there exists another element $z \in G(\Sigma)$ such that $z_1+z_2=z$ (see for example \cite{Cas} page 1480).

By Theorem \ref{primitive}, the assumption implies that there is a primitive collection $\P=\{x_1, \cdots, x_{n}\}$ such that
$x_1+\cdots+x_n=0$. The vector space $H:=\rit x_1 +\cdots+\rit x_n$ divides $N \otimes_\zit \rit$ into two sides. Assume that on one side, we have at least three elements $y_1, y_2, y_3$ in $G(\Sigma)$.
 Take an element $z \in G(\Sigma)$ on the other side of $H$, then $\{z, y_1\}$, $\{z, y_2\}$, $\{z, y_3\}$ are all primitive collections. By Lemma \ref{2prim}, up to reordering, we have $z+y_1=0$,  $z+y_2 = (-y_3)$ and $-y_2, -y_3$
 are all elements in $G(\Sigma)$.  This gives $y_1=y_2+y_3$. We consider the primitive collections $\{-y_2, y_1\}$,
 $\{-y_2, y_3\}$. By Lemma \ref{2prim}, we obtain that $-y_2 + y_1 = -y_3$, which contradicts $y_1=y_2+y_3$.
Thus there exist at most two elements on each side of $H$. Let $y_1, y_2$ be the two elements on one side
of $H$ and $z_1, z_2$ on the other side. If $z_i$ is not in $\{-y_1, -y_2\}$, then we may apply Lemma \ref{2prim},
which shows that $-y_1, -y_2$ are in $G(\Sigma)$, a contradiction. Up to reordering, we may assume
$z_1 = -y_1$ and $z_2 = -y_2$. Consider the primitive collections $\{-y_1, y_2\}$ and $\{-y_2, y_1\}$.
Their primitive relations are $-y_1+y_2 = x_i, -y_2+y_1 = x_j$ for some $i, j$.
This implies that $x_i+x_j=0$, which is not possible since $n \geq 3$.
In conclusion, the set $G(\Sigma)$ has at most  $n+3$ elements, while $\rho_X = \sharp G(\Sigma) -n$.   As a consequence, we have $\rho_X \leq 3$.

If $\rho_X=3$, there are two elements $y_1, y_2$ on one side of $H$ and  an element $z$ on the other side.
Up to reordering, the precedent argument shows that $z=-y_1$ and $-y_1+y_2=x_1$, i.e. $y_1+x_1=y_2$.
 This shows that $X$ is the  blow-up of $\pit^{n-1} \times \pit^1$ along the invariant subvariety isomorphic to  $\pit^{n-2}$  corresponding to the cone generated by $x_1, y_1$.

If $\rho_X=2$, i.e. $G(\Sigma)$ has $n+2$ elements. On each side
of $H$, there exists exactly one element in $G(\Sigma)$, say, $y$
and $ z$. As $\{y, z\}$ is a primitive collection, one has either
$y+z=0$ or $y+z =x_i$ for some $i$. The first case corresponds to
$\pit^{n-1} \times \pit^{1}$ while the second fan corresponds to
$\pit(\0_{\pit^1}^{\oplus{n-1}} \oplus \0_{\pit^1}(1))$.
\end{proof}

\section{A conjecture}
For a toric Fano manifold $X$ of dimension $n$, the pseudo-index $\iota_X$ is the smallest intersection number
$-K_X \cdot C$ among all rational curves on $X$. In Theorem 1 \cite{Ca2}, it is proven that $\rho_X \leq 2n$ and $\rho_X (\iota_X -1) \leq n$,
which confirms a conjecture of Mukai.
In analogous to this, we would like to propose the following
\begin{Conj} \label{Conj}
For a toric Fano manifold $X^n$ with $n \geq 3$, if there exists a minimal component $\K$ of degree $p+2$, then $\rho_X \cdot (p+1) \leq n(n+1)/2$.
\end{Conj}
Note that the equality holds if  $X \simeq (S_3)^d \times \pit^{2d}$ or $(S_3)^d \times \pit^{2d+1}$, where $S_3$ is the blow-up of $\pit^2$ along three general points.
In dimension 3, the equality also holds for the blow-up of $\pit^2 \times \pit^1$ along a $\pit^1$ contained in $\pit^2$.

Since $\rho_X \leq 2n$ by \cite{Ca2}, we may assume that $p+1 >\frac{n+1}{4}$ to check Conjecture \ref{Conj}.
When $n=3$, Conjecture \ref{Conj} is immediate from Proposition \ref{degn}.
To check  Conjecture \ref{Conj} for $n=4$, one only needs to
 show that if  $p=1$, then $\rho_X \leq 5$. We can use the classification of 4-dimensional toric Fano manifolds
 in \cite{Ba2} to check that if $\rho_X \geq 6$, then $p=0$, which shows that our conjecture holds for $n= 4$.

\bigskip
Baohua Fu

Institute of Mathematics, AMSS, 55 ZhongGuanCun East Road,

Beijing, 100190, China

 bhfu@math.ac.cn

\bigskip
Jun-Muk Hwang

 Korea Institute for Advanced Study, Hoegiro 87,

Seoul, 130-722, Korea

jmhwang@kias.re.kr

\end{document}